\documentclass[11pt]{amsart}
\usepackage{amsmath}
\usepackage[utf8]{inputenc}
\usepackage{amssymb}
\usepackage{amsopn}
\usepackage{epsfig}
\usepackage{amsfonts}
\usepackage{mathrsfs}
\usepackage{latexsym}
\usepackage{graphicx}
\usepackage{enumerate}
\usepackage{color}
\usepackage{tikz, subfigure}
\newtheorem{theorem}{Theorem}[section]
\newtheorem{lemma}[theorem]{Lemma}
\newtheorem{proposition}[theorem]{Proposition}
\newtheorem{corollary}[theorem]{Corollary}
\theoremstyle{definition}

\newtheorem{question}{Question}[section]

\theoremstyle{remark}
\newtheorem{remark}[theorem]{Remark}

\numberwithin{equation}{section}

\newcommand{\R}{\ensuremath{\mathbb{R}}}
\newcommand{\N}{\ensuremath{\mathbb{N}}}

\newcommand{\set}[1]{\left\{#1\right\}}

\newcommand{\ep}{\varepsilon}
\newcommand{\f}{\infty}

\begin{document}
\title[Intersections with Cantor sets]{On the Intersection of Cantor set with the unit circle and some sequences}
\author{Kan Jiang}
\address[K. Jiang]{Department of Mathematics, Ningbo University, Ningbo 315211, People's Republic of China}
\email{jiangkan@nbu.edu.cn}

\author{Derong Kong}
\address[D. Kong]{College of Mathematics and Statistics, Center of Mathematics, Chongqing University, Chongqing 401331, People's Republic of China}
\email{derongkong@126.com}

\author{Wenxia Li}
\address[W. Li]{School of Mathematical Sciences, Key Laboratory of MEA (Ministry of Education) \& Shanghai Key Laboratory of PMMP, East China Normal University, Shanghai 200241, People's Republic of China}
\email{wxli@math.ecnu.edu.cn}

\author{Zhiqiang Wang}
\address[Z. Wang]{College of Mathematics and Statistics, Center of Mathematics \& Key Laboratory of Nonlinear Analysis and its Applications (Ministry of Education), Chongqing University, Chongqing 401331, People's Republic of China}
\email{zhiqiangwzy@163.com,zqwangmath@cqu.edu.cn}

\date{\today}
\begin{abstract}
For $\lambda\in(0,1/2)$ let $K_\lambda$ be the self-similar set in $\R$ generated by the  iterated function system $\set{f_0(x)=\lambda x, f_1(x)=\lambda x+1-\lambda}$.
In this paper, we investigate the intersection of the unit circle $\mathbb{S} \subset \mathbb{R}^2$ with the Cartesian product $K_{\lambda} \times K_{\lambda}$. We prove that for $\lambda \in(0, 2 - \sqrt{3}]$, the intersection is \emph{trivial}, i.e.,
\[
\mathbb{S} \cap (K_{\lambda} \times K_{\lambda}) = \{(0,1), (1,0)\}.
\]
If $\lambda\in [0.330384,1/2)$, then the intersection $\mathbb{S} \cap (K_{\lambda} \times K_{\lambda})$ is non-trivial. In particular, if $\lambda\in [0.407493 , 1/2)$ the intersection $\mathbb{S} \cap (K_{\lambda} \times K_{\lambda})$ is of cardinality continuum.
Furthermore, the bound $2 - \sqrt{3}$ is sharp: there exists a sequence $\{\lambda_n\}_{n \in \mathbb{N}}$ with $\lambda_n \searrow 2 - \sqrt{3}$ such that
$\mathbb{S} \cap (K_{\lambda_n} \times K_{\lambda_n})$ is non-trivial for all $n\in\N$.
This result provides a negative answer to a problem posed by Yu (2023).
Our methods extend beyond the unit circle and remain effective for many nonlinear curves. By employing tools from number theory, including the quadratic reciprocity law, we   analyze the intersection  of   Cantor sets with some sequences.  A dichotomy is established in terms of the Legendre symbol associated with the digit set, revealing a fundamental arithmetic constraint governing such intersections.
\end{abstract}
\maketitle

\section{Introduction}
Given a continuous curve $\Gamma$ and a fractal set $\mathcal{K}$ on the plane, determining the closed form of their intersection is in general a challenging problem.
The classical Marstrand's slicing theorem \cite{Mattila1999} describes the Hausdorff dimension of the intersection $\Gamma\cap \mathcal{K}$ for a typical line $\Gamma$. However, the intersection is difficult to analyze if the curve $\Gamma$ is nonlinear.
Analogous problems can be considered in the discrete case. Given a sequence $\{x_n\}_{n=1}^\f$ and a Cantor set $K$ in $\mathbb R$, how can we describe the intersection $\{x_n:n\in \mathbb{N}\}\cap K$?
For instance, we even do not know the structures of the following intersections:
$$\left\{\dfrac{1}{n^2}:n\in \mathbb{N}\right\}\cap C \quad\mbox{and}\quad \left\{\dfrac{1}{n!}:n\in \mathbb{N}\right\}\cap C.$$
Here and throughout the paper, $C$ stands for the middle-third Cantor set.

Another motivation comes from a recent question posed by Yu \cite{Yu2023}, where he made use of the $\ell^1$-dimension to study the distribution of missing digits points near manifolds.
Yu posed the following general question \cite[Question 1.5]{Yu2023}.

\begin{question}\label{Yu-general-question}
Let $M$ be a non-degenerate analytic manifold, and let $\mathcal{K}$ be a missing digits set in $\R^d$. To determine whether or not $M \cap \mathcal{K}$ is infinite.
\end{question}

For $\lambda\in(0,1/2)$,  let $K_\lambda$ be the self-similar set generated by the iterated function system (IFS) (cf. \cite{Hutchinson1981}) $$ \big\{ f_0(x) = \lambda x ,\; f_1(x) = \lambda x + 1- \lambda \big\}. $$
Then the convex hull of $K_\lambda$ is the unit interval $[0,1]$ for all $\lambda\in(0,1/2)$. A particular special case of Question \ref{Yu-general-question} was formulated by Yu \cite{Yu2023} as follows.
\begin{question}\label{Yu-special-question}
Consider the unit circle $$\mathbb S:=\set{(x,y)\in\mathbb R^2:x^2+y^2=1}.$$ Is the intersection $\mathbb S\cap (K_{1/5}\times K_{1/5})$ infinite?
\end{question}
Yu remarked that the methods developed by Shmerkin \cite{Shmerkin-2019} and Wu \cite{Wu-2019} may be used to deduce that the points under consideration form a set with zero Hausdorff dimension. But, this is not sufficient to conclude the finiteness. Note that $K_{1/3}=C$ is the middle-third Cantor set.
Recently, with the assistance of computers, Du, Jiang and Yao \cite{DJY2024} proved that the intersection $\mathbb S\cap (K_{1/3}\times K_{1/3})$ contains  at least $10,000,000$ points.

In this paper we address two aspects. On  one hand, for a Cantor set $K\subset\mathbb R$ we consider the intersection of $K\times K$ with a curve, and give a negative answer to Question \ref{Yu-special-question} and  partially resolve Question \ref{Yu-general-question}.
On the other hand, we investigate the intersection of $K$ with  a sequence of real numbers.
Throughout the paper, let $\mathbb{N}$ denote the set of all positive integers, and for $k\in \mathbb{N}$, we define $\mathbb{N}_{\geq k}=[k,+\infty)\cap \mathbb{N}$.
For $a,b\in \mathbb{N}$ we write $\gcd(a,b)$ for the greatest common divisor of $a$ and $b$.
We use $\# A$ to denote the cardinality of a set $A$.

Note that for each $\lambda\in(0,1/2)$ the self-similar set $K_\lambda$ is a Cantor set in $\R$.
Our first main result focuses on the intersection  $\mathbb S\cap(K_\lambda \times K_\lambda)$.
We say the intersection $\mathbb S\cap(K_\lambda \times K_\lambda)$  is \emph{trivial} if it consists of exactly two trivial points $(0,1)$ and $(1,0)$; otherwise, we say  $\mathbb S\cap(K_\lambda \times K_\lambda)$ is \emph{non-trivial}.

\begin{theorem}\label{circle-lambda}\mbox{}
\begin{enumerate}[{\rm(i)}]
\item For $0< \lambda \le 2 - \sqrt{3}$, the intersection $\mathbb S\cap(K_\lambda \times K_\lambda)$ is trivial.

\item For $0.330384 \le \lambda < 1/2$, the intersection $\mathbb S\cap(K_\lambda \times K_\lambda)$ is non-trivial.

\item For $0.407493 \le \lambda < 1/2$, the intersection $\mathbb S\cap(K_\lambda \times K_\lambda)$ is of cardinality continuum.
\end{enumerate}
\end{theorem}
\begin{remark}
\begin{enumerate}[{\rm(a)}]
\item Note that $0<1/5<2-\sqrt{3}$. So, by Theorem \ref{circle-lambda} (i) it follows that $$\mathbb S\cap(K_{1/5}\times K_{1/5}) = \{(0,1),(1,0)\}.$$
This provides a negative answer to Question \ref{Yu-special-question}.

\item For $k \in \mathbb{N}$, take $x= \lambda^k(1-\lambda)$ and $y= 1-\lambda^{2k+1}$ in $K_\lambda$. Define $\lambda_k \in(0, 1/2)$ to be the appropriate  root of the equation $\lambda^{2k}(1-\lambda)^2 + (1-\lambda^{2k+1})^2 = 1$, i.e., $\lambda^{2k+2} + \lambda^2-4\lambda+1=0$.
Note that the intersection $\mathbb S\cap(K_{\lambda_k}\times K_{\lambda_k})$ is non-trivial, and the sequence $\{\lambda_k\}$ is decreasing to $2-\sqrt{3}$ as $k \to \f$.
Thus, the constant $2-\sqrt{3}$ in Theorem \ref{circle-lambda} (i) is optimal.

\item The constants in Theorem \ref{circle-lambda} (ii) and (iii) is not optimal.
We conjecture that the intersection $\mathbb S^1\cap(K_{\lambda}\times K_{\lambda})$    is infinite for all $2-\sqrt{3} < \lambda < 1/2$.
\end{enumerate}\end{remark}

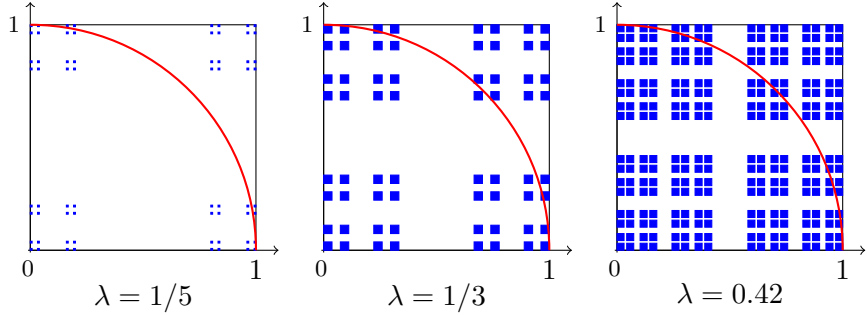
\begin{figure}[h]

\begin{tikzpicture}[scale=3]
\draw[->](-.01,0)node[below]{\footnotesize$0$}--(1,0)node[below]{$1$}--(1.1,0);
\draw[->](0,-.01)--(0,1)node[left]{\footnotesize $1$}--(0,1.1);
\draw(0,0) rectangle (1,1);

\foreach \x in {0,0.2^2-0.2^3, 0.2-0.2^2, 0.2-0.2^3, 1-0.2, 1-0.2+0.2^2-0.2^3, 1-0.2^2, 1-0.2^3}
\foreach \y in {0,0.2^2-0.2^3, 0.2-0.2^2, 0.2-0.2^3, 1-0.2, 1-0.2+0.2^2-0.2^3, 1-0.2^2, 1-0.2^3}
{\filldraw[blue](\x,\y) rectangle (\x+0.2^3,\y+0.2^3);}

\draw[thick,red](1,0) arc [start angle=0, end angle=90, radius=1cm];

\draw (0.5,-0.1)node[below]{$\lambda=1/5$};

\draw[->](1.3-.01,0)node[below]{\footnotesize$0$}--(1.3+1,0)node[below]{$1$}--(1.3+1.1,0);
\draw[->](1.3,-.01)--(1.3,1)node[left]{\footnotesize $1$}--(1.3,1.1);
\draw(1.3,0) rectangle (1.3+1,1);

\foreach \x in {0,1/3^2-1/3^3, 1/3-1/3^2, 1/3-1/3^3, 1-1/3, 1-1/3+1/3^2-1/3^3, 1-1/3^2, 1-1/3^3}
\foreach \y in {0,1/3^2-1/3^3, 1/3-1/3^2, 1/3-1/3^3, 1-1/3, 1-1/3+1/3^2-1/3^3, 1-1/3^2, 1-1/3^3}
{\filldraw[blue](1.3+\x,\y) rectangle (1.3+\x+1/3^3,\y+1/3^3);}

\draw[thick,red](1.3+1,0) arc [start angle=0, end angle=90, radius=1cm];

\draw (1.3+0.5,-0.1)node[below]{$\lambda=1/3$};

\draw[->](2.6-.01,0)node[below]{\footnotesize$0$}--(2.6+1,0)node[below]{$1$}--(2.6+1.1,0);
\draw[->](2.6,-.01)--(2.6,1)node[left]{\footnotesize $1$}--(2.6,1.1);
\draw(2.6,0) rectangle (2.6+1,1);

\foreach \x in {0,0.42^3-0.42^4, 0.42^2-0.42^3, 0.42^2-0.42^4, 0.42-0.42^2, 0.42-0.42^2+0.42^3-0.42^4, 0.42-0.42^3, 0.42-0.42^4, 1-0.42, 1-0.42+0.42^3-0.42^4, 1-0.42+0.42^2-0.42^3,1-0.42+0.42^2-0.42^4, 1-0.42^2, 1-0.42^2+0.42^3-0.42^4, 1-0.42^3, 1-0.42^4}
\foreach \y in {0,0.42^3-0.42^4, 0.42^2-0.42^3, 0.42^2-0.42^4, 0.42-0.42^2, 0.42-0.42^2+0.42^3-0.42^4, 0.42-0.42^3, 0.42-0.42^4, 1-0.42, 1-0.42+0.42^3-0.42^4, 1-0.42+0.42^2-0.42^3,1-0.42+0.42^2-0.42^4, 1-0.42^2, 1-0.42^2+0.42^3-0.42^4, 1-0.42^3, 1-0.42^4}
{\filldraw[blue](2.6+\x,\y) rectangle (2.6+\x+0.42^4,\y+0.42^4);}

\draw[thick,red](2.6+1,0) arc [start angle=0, end angle=90, radius=1cm];

\draw (2.6+0.5,-0.1)node[below]{$\lambda=0.42$};
\end{tikzpicture}
{\caption{The intersection   $\mathbb S^1\cap(K_\lambda \times K_\lambda)$ with $\lambda=1/5, 1/3$ and $0.42$.}}
\label{}
\end{figure}

Next, we consider the intersection of $K_\lambda\times K_{\lambda}$ with some other curves or surfaces.

\begin{theorem}\label{curve}\mbox{}
\begin{enumerate}[{\rm(i)}]
\item Let $0< \lambda \le 1/3$, and let $q$ be an integer with $2 \le q \le 1/\lambda -1$. Then
$$\{(x,y):y=x^{q}\}\cap (K_\lambda \times K_\lambda) =\big\{ (\lambda^k,\lambda^{qk} ): k\in \mathbb{N}\big\}\cup\{(0,0), (1,1)\}. $$

\item  For $0< \lambda \le 0.187915$, we have $$\{(x,y,z):  x^2+y^2=z^2\}\cap (K_\lambda\times K_\lambda \times K_\lambda) =\{(x,0,x), (0,x,x): x \in K_\lambda\}. $$
\end{enumerate}
\end{theorem}
\begin{remark}
A direct corollary of Theorem \ref{curve} (i) is that
$$\{(x,y):y=x^2\}\cap (K_{1/3} \times K_{1/3}) =\big\{ (3^{-k},9^{-k} ): k\in \mathbb{N}\big\}\cup\{(0,0), (1,1)\}. $$
\end{remark}

Finally, we investigate the intersection of Cantor set with some sequences.
We will employ some ideas from $q$-expansions and quadratic residues from number theory. Let $p\in \N$ be an odd prime, and let $a \in \mathbb{Z}$ with $p \nmid a$. We use
\[ \left(\dfrac{a}{p}\right)_L \]
to denote the Legendre's symbol, see Section \ref{sec:sequence} for more details.

Let $m\in\N_{\ge 3}$ and $D \subset\set{0,1,\ldots, m-1}$ with $1< \#D < m$. We define
$$
K_{m,D}=\left\{ \sum _{k=1}^\infty \frac{d_k}{m^k}: d_k\in D~\forall k\in \N \right\}.
$$
The set $K_{m,D}$ is a (homogeneous) self-similar set generated by the IFS $\{\varphi_d(x)=(x+d)/m: d\in D\}$.
Note that $K_{3,\set{0,2}}=C$ is the middle-third Cantor set.
The structure of rational points in Cantor set $K_{m,D}$ has been studied in \cite{Bloshchitsyn-2015, JKLW2024, Li-Li-Wu-2022, Nagy-2001, Schleischitz-2021, Shparlinski-2021, Wall-1990}.
For $p \in \mathbb{N}_{\ge 2}$ let $\mathcal{D}_p$ be the set of all rational numbers in $[0,1]$ having a finite $p$-ary expansion.
If $\gcd(p,m)=1$, Schleischitz \cite[Corollary 4.4]{Schleischitz-2021} showed that the intersection $\mathcal{D}_p \cap K_{m,D}$ is a finite set. In the following we give a complete description on the intersection of $K_{m, D}$ with the sequence $\set{1/n^2: n\in\N}$.

\begin{theorem}\label{sequence}
Let $m \in \mathbb{N}_{\geq 3}$ be a prime, and $D\subset \{0,1,\ldots, m-1\}$ with $0\in D$ and $1<\#D<m$.
Suppose that $$\left(\dfrac{-a}{m}\right)_L=-1\quad \forall a\in D\setminus\set{0}.$$
Then we have
\begin{itemize}
\item[(1)] if $1\notin D$, $$\left\{\dfrac{1}{n^{2}}:n\in \mathbb{N}\right\}\cap K_{m,D} =\emptyset;$$
\item[(2)] if $1\in D$, $$\left\{\dfrac{1}{n^{2}}:n\in \mathbb{N}\right\}\cap K_{m,D} =\left\{\dfrac{1}{m^{2\ell}}: \ell \in \mathbb{N} \right\}.$$
\end{itemize}
\end{theorem}
\begin{remark}
\begin{enumerate}[{\rm(a)}]
\item By \cite[Theorem 84]{Hardy2008}, there are $(p-1)/2$ quadratic residues and $(p-1)/2$ quadratic non-residues of an odd prime $p$ in the set $\{1,2,\ldots,p-1\}$. So, the cardinality of the set $D$ in Theorem \ref{sequence} can be at most $(m+1)/2$.
      By taking a sufficiently large odd prime $p$, applying Theorem \ref{sequence} (i) we can obtain that for any $\epsilon>0$, there exists a self-similar set $K_\ep \subset [0,1]$ with $0 = \min K_\ep$ such that $\dim_{H} K_\ep > 1-\epsilon$ and $$\bigg\{\dfrac{1}{n^2}:n\in \mathbb{N}\bigg\}\cap K_\ep =\emptyset.$$

\item  As a direct corollary of Theorem \ref{sequence}, we have $$\left\{\dfrac{1}{n^{2}}:n\in \mathbb{N}\right\}\cap K_{3,\{0,1\}} =\left\{\dfrac{1}{9^{\ell}}: \ell \in \mathbb{N} \right\}.$$
That is, $$\left\{\dfrac{2}{n^{2}}:n\in \mathbb{N}\right\}\cap C =\left\{\dfrac{2}{9^{\ell}}: \ell \in \mathbb{N} \right\}.$$


\item  
Using the almost identical argument in the proof of Theorem \ref{sequence}, we can deal with cubic sequence in the following concrete example:
$$\bigg\{ \frac{1}{n^3}: n \in \mathbb{N} \bigg\} \cap K_{7,\{0,2,3,4,5\}} = \emptyset.$$
\end{enumerate}
\end{remark}

The paper is organized as follows.
In Section \ref{sec:unit-circle}, we focus on the intersection with the unit circle and prove Theorem \ref{circle-lambda}.
We will prove Theorem \ref{curve} in Section \ref{sec:curve}.
Section \ref{sec:sequence} is dedicated to proving Theorem \ref{sequence} and some corollaries. Finally, we  list some questions  in Section \ref{sec:question}.

\section{Intersection with the unit circle}\label{sec:unit-circle}

In this section, we will prove Theorem \ref{circle-lambda}. Recall that $K_\lambda$ is a self-similar set generated by the IFS
$\big\{ f_0(x) = \lambda x ,\; f_1(x) = \lambda x + 1- \lambda \big\}$ for $\lambda \in (0,1/2)$.
For $\mathbf{i}=i_1i_2 \ldots i_n \in \{0,1\}^n$, define \[ f_{\mathbf{i}} = f_{i_1} \circ f_{i_2} \circ \cdots \circ f_{i_n}. \]
For $n\in \N$, the $n$-level basic intervals of $K_\lambda$ are defined by
$$\mathcal{F}_n:= \big\{ f_{\mathbf{i}}([0,1]): \mathbf{i} \in \{0,1\}^n \big\}.$$
For $I \in \mathcal{F}_n$ and $k \ge n$, we also define
$$\mathcal{F}_k^I:= \big\{ J: J\in \mathcal{F}_k,\; J \subset I \big\}.$$
For a collection of sets $\mathcal{F}$, let $\bigcup \mathcal{F}$ denote the union of all sets in $\mathcal{F}$.
Then we have $$K_\lambda = \bigcap_{n=1}^\f \bigcup \mathcal{F}_n  \quad\text{and}\quad K_\lambda \cap I = \bigcap_{k=n}^\f \bigcup \mathcal{F}_k^I.$$

Recall that $\mathbb S=\set{(x,y): x^2+y^2=1}$ is the unit circle. The proof of Theorem \ref{circle-lambda} will be split into three propositions.

\begin{proposition}\label{prop:circle-1}
For $0< \lambda \le 2 - \sqrt{3}$, we have $$\mathbb S\cap (K_\lambda\times K_\lambda) =\{(0,1), (1,0)\}. $$
\end{proposition}
\begin{proof}
  Take a point $(x,y)\in\mathbb S\cap(K_\lambda\times K_\lambda)$.  By symmetry, we assume $x \le y$. It suffices to show $x=0$.

  First, we show $x \in f_0([0,1])$.
  Observe that $$x \in K_\lambda = f_0(K_\lambda) \cup f_1(K_\lambda) \subset f_0([0,1]) \cup f_1([0,1]).$$
  Suppose on the contrary that $x \in f_1([0,1])$. Then $x \ge 1-\lambda$, and hence we have $$x^2 + y^2 \ge 2 x^2 \ge 2(1-\lambda)^2\ge 2(\sqrt{3}-1)^2 > 1.$$
  This leads to a contradiction with $x^2 + y^2 =1$.

  Next, assuming $x \in f_{0^k}([0,1])$ for some $k \in \N$ we show $x\in f_{0^{k+1}}([0,1])$.
  If this were done, then by induction we would conclude that $x\in f_{0^\ell}([0,1])$ for all $\ell \in \N$, and thus $x=0$.

  Notice that $$x\in f_{0^{k}}([0,1]) \cap K_\lambda  \subset f_{0^{k+1}}([0,1]) \cup f_{0^{k}1}([0,1]).$$
  Suppose on the contrary that $x \in f_{0^{k}1}([0,1])$. Then we have $$\lambda^k(1-\lambda) \le x \le \lambda^k.$$
  If $y \in f_{1^{2k+1}}([0,1])$, then we have $y\ge 1-\lambda^{2k+1}$, and $$x^2 + y^2 \ge \lambda^{2k}(1-\lambda)^2 + (1-\lambda^{2k+1})^2 = 1 + \lambda^{4k+2} + (\lambda^2 - 4\lambda +1)\lambda^{2k} > 1,$$
  where the last inequality follows from $0< \lambda \le 2 - \sqrt{3}$.
  If $y \not\in f_{1^{2k+1}}([0,1])$, then we have $y \le 1- \lambda^{2k} + \lambda^{2k+1}$ because $y \in K_\lambda$.
  Using $0< \lambda \le 2 - \sqrt{3}$, we get
  \begin{align*}
    x^2 + y^2
    & \le \lambda^{2k} + (1-\lambda^{2k}+ \lambda^{2k+1})^2 \\
    & = 1-(1-2\lambda)\lambda^{2k} + (1-\lambda)^2 \lambda^{4k} \\
    & < 1- (1-2\lambda)\lambda ^{2k}+(1-\lambda )^2 \lambda^{2k+1} \\
    & < 1.
  \end{align*}
  Thus, we obtain $x^2 + y^2 \ne 1$, a contradiction.
  Hence, we conclude $$x \in f_{0^{k+1}}([0,1]),$$ completing the proof.
\end{proof}

Note that $\mathcal F_n$ consists of $2^n$ pairwise disjoint intervals, and   each of them is of the form $[a, a+\lambda^n]$ with
\[
a\in\bigg\{(1-\lambda)\sum_{j=1}^{n}d_j\lambda^{j-1}: d_j\in\set{0, 1}~\forall 1\le j\le n \bigg\}.
\]
In the remainder of this section, we always assume that $2-\sqrt{3} < \lambda < 1/2$, and write $g(x,y) = x^2 + y^2$.
\begin{lemma}\label{lem:check-condition}
Let $I,J \in \mathcal{F}_n$ with $I=[a,a+\lambda^n]$, $J=[b, b+\lambda^n]$, and $a \le b$.
Suppose that
$$\frac{1-2\lambda}{\lambda} (a+\lambda^n) \le  b \le \frac{a}{1-2\lambda}.$$
Then we have $$g(I\times J) = g\big( (I_1 \cup I_2) \times(J_1 \cup J_2) \big),$$
where $$I_1 = [a, a + \lambda^{n+1}],\quad I_2 = [a+(1-\lambda)\lambda^n, a+\lambda^n ],$$
  and
  $$J_1 = [b, b + \lambda^{n+1}],\quad J_2 = [b+(1-\lambda)\lambda^n, b+\lambda^n ].$$
\end{lemma}
Before the proof we point out that $\mathcal F_{n+1}^I=\set{I_1, I_2}$ and $\mathcal F_{n+1}^J=\set{J_1, J_2}$.
\begin{proof}
  It is easy to calculate that
  \begin{align*}
    g(I_1\times J_1) & = [a^2+b^2, a^2+b^2+ 2(a+b)\lambda^{n+1} + 2\lambda^{2n+2}], \\
    g(I_2 \times J_1) & = [a^2+b^2+ 2a (1-\lambda)\lambda^n +(1-\lambda)^2\lambda^{2n}, \\
    & \hspace{8em} a^2+b^2+ 2(a+b\lambda)\lambda^{n} + (1+\lambda^2)\lambda^{2n}], \\
    g(I_1\times J_2) & = [a^2+b^2+ 2b (1-\lambda)\lambda^n +(1-\lambda)^2\lambda^{2n}, \\
    & \hspace{8em} a^2+b^2+ 2(a\lambda+b)\lambda^{n} + (1+\lambda^2)\lambda^{2n}], \\
    g(I_2\times J_2) & = [a^2+b^2+ 2(a+b) (1-\lambda)\lambda^n + 2(1-\lambda)^2\lambda^{2n}, \\
    & \hspace{12em} a^2+b^2+ 2(a+b)\lambda^{n} + 2\lambda^{2n}].
  \end{align*}
  It remains to verify that the above four intervals are overlapping, which is equivalent to
  \begin{equation}\label{eq:equivalent-condition}
  \left\{
    \begin{array}{l}
      -(\lambda-1)^2\lambda^n \le 2b\lambda - 2(a+\lambda^n)(1-2\lambda),  \\
      -\lambda^{n+1} \le a- b(1-2\lambda),  \\
      (1-4\lambda+\lambda^2)\lambda^n \le 2b \lambda - 2a(1-2\lambda).
    \end{array}
  \right.
  \end{equation}
For $2-\sqrt{3}< \lambda < 1/2$, the left-hand sides of all three inequalities in (\ref{eq:equivalent-condition}) are negative. Thus, we only need
$$b\lambda - (a+\lambda^n)(1-2\lambda) \ge 0 \quad\textrm{and}\quad a- b(1-2\lambda) \ge 0.$$
That is, $$\frac{1-2\lambda}{\lambda} (a+\lambda^n) \le  b \le \frac{a}{1-2\lambda},$$
as desired.
\end{proof}

The following lemma is an easy exercise in real analysis, and we only state it without a detailed proof.
\begin{lemma}\label{lem:continuity}
Let $f:\R^2 \to \R$ be a continuous function, and let $\{F_n\}_{n=1}^\f$ be a decreasing sequence of nonempty compact subsets of $\R^2$. Then we have
$$f\bigg( \bigcap_{n=1}^\f F_n \bigg) = \bigcap_{n=1}^\f f(F_n).$$
\end{lemma}

\begin{lemma}\label{lem:find-nontrivial}
Let $I,J \in \mathcal{F}_n$ with $I=[a,a+\lambda^n]$, $J=[b, b+\lambda^n]$, and $a \le b$.
Suppose that
$$\frac{1-2\lambda}{\lambda}(a+\lambda^n) \le b < b+\lambda^n \le \frac{a}{1-2\lambda}.$$
Then we have $$g\big( (K_\lambda \cap I) \times (K_\lambda \cap J) \big) = g(I\times J).$$
\end{lemma}
\begin{proof}
Fix $k \ge n$ and take $I'=[a',a'+\lambda^k] \in \mathcal{F}_k^{I}$ and  $J'=[b',b'+\lambda^k] \in \mathcal{F}_k^{J}$ with $a' \le b'$.
Note that $a\le a' \le a+\lambda^n-\lambda^k$ and $b\le b' \le b+\lambda^n-\lambda^k$.
Using our assumptions, we clearly have $$\frac{1-2\lambda}{\lambda}(a'+\lambda^k) \le \frac{1-2\lambda}{\lambda}(a+\lambda^n) \le b\le b',$$
and
$$b' <b+\lambda^n \le \frac{a}{1-2\lambda}\le \frac{a'}{1-2\lambda}.$$
Thus, by Lemma \ref{lem:check-condition} we obtain
\begin{equation}\label{eq:induction-g}
  g(I' \times J') = g\Big( \big(\bigcup \mathcal{F}_{k+1}^{I'}\big) \times \big(\bigcup \mathcal{F}_{k+1}^{J'}\big) \Big).
\end{equation}

If $a < b$, then for any $I'=[a',a'+\lambda^k] \in \mathcal{F}_k^{I}$, $J'=[b',b'+\lambda^k] \in \mathcal{F}_k^{J}$ we have $a' < b'$.
Thus, by (\ref{eq:induction-g}) we obtain
\begin{align*}
g\bigg( \Big(\bigcup \mathcal{F}_k^{I} \Big) \times \Big(\bigcup\mathcal{F}_k^{J}\Big) \bigg) & = \bigcup_{ I' \in \mathcal{F}_k^{I} } \bigcup_{ J' \in \mathcal{F}_k^{J} } g(I' \times J') \\
& = \bigcup_{ I' \in \mathcal{F}_k^{I} } \bigcup_{ J' \in \mathcal{F}_k^{J} } g\Big( \big(\bigcup \mathcal{F}_{k+1}^{I'}\big) \times \big(\bigcup \mathcal{F}_{k+1}^{J'}\big) \Big) \\
& = g\bigg( \Big(\bigcup \mathcal{F}_{k+1}^{I} \Big) \times \Big(\bigcup\mathcal{F}_{k+1}^{J}\Big) \bigg).
\end{align*}
If $a = b$, then $I=J$ and let $\mathscr{F}_k$ be the collection of pairs $(I',J') \in \mathcal{F}_k^{I} \times \mathcal{F}_k^{J}$ with $I'=[a',a'+\lambda^k]$, $J'=[b',b'+\lambda^k]$, and $a' \le b'$.
Note that $g(x,y) = g(y,x)$.
By (\ref{eq:induction-g}) We also have
\begin{align*}
g\bigg( \Big(\bigcup \mathcal{F}_k^{I} \Big) \times \Big(\bigcup\mathcal{F}_k^{J}\Big) \bigg)
& = \bigcup_{ (I',J') \in \mathscr{F}_k } g(I' \times J') \\
& = \bigcup_{ (I',J') \in \mathscr{F}_k } g\Big( \big(\bigcup \mathcal{F}_{k+1}^{I'}\big) \times \big(\bigcup \mathcal{F}_{k+1}^{J'}\big) \Big) \\
& = g\bigg( \Big(\bigcup \mathcal{F}_{k+1}^{I} \Big) \times \Big(\bigcup\mathcal{F}_{k+1}^{J}\Big) \bigg).
\end{align*}

Finally, by Lemma \ref{lem:continuity}, we conclude that
\begin{align*}
g(I\times J) & = \bigcap_{k=n}^\f g\bigg( \Big(\bigcup \mathcal{F}_k^{I} \Big) \times \Big(\bigcup\mathcal{F}_k^{J}\Big) \bigg) \\
& = g\bigg( \bigcap_{k=n}^\f \Big(\bigcup \mathcal{F}_k^{I} \Big) \times \Big(\bigcup\mathcal{F}_k^{J}\Big) \bigg) \\
& = g\Big( \big( K_\lambda \cap I \big) \times \big( K_\lambda \cap J \big) \Big),
\end{align*}
as desired.
\end{proof}

\begin{proposition}\label{prop:circle-2}
  For $0.330384 \le \lambda < 1/2$, we have $$\# \big( \mathbb{S} \cap (K_\lambda\times K_\lambda) \big) >2 . $$
\end{proposition}
\begin{proof}
Suppose that $I,J \in \mathcal{F}_n$ satisfy all the conditions in Lemma \ref{lem:find-nontrivial}, and moreover, we have $1\in g(I,J)$.
Then by Lemma \ref{lem:find-nontrivial}, we have $1\in g\big( ( K_\lambda \cap I ) \times ( K_\lambda \cap J ) \big)$.
This means that there exists $(x,y) \in (K_\lambda \cap I) \times (K_\lambda \cap J)$ such that $x^2+y^2=1$.
Thus, we only need to find desired $n$-level basic intervals $I,J \in \mathcal{F}_n$ for some $n \in \N$.

(i) For $9/25 \le \lambda < 1/2$, choose $I=J=[1-\lambda, 1-\lambda+\lambda^2]$.
That is, $n=2$ and $a=b= 1-\lambda$ in Lemma \ref{lem:find-nontrivial}.
Note that $2(1-\lambda)^2 \le 512/625 < 1$ and $2(1-\lambda+\lambda^2)^2 \ge 9/8 > 1$.
This implies that $1 \in g(I\times J)$.
We clearly have $b+\lambda^2 < 1 < a/(1-2\lambda)$.
By Lemma \ref{lem:find-nontrivial}, it remains to verify
$$\frac{1-2\lambda}{\lambda}(1-\lambda+\lambda^2)\le 1-\lambda,$$
i.e., $2\lambda^3 -4\lambda^2 +4\lambda -1 \ge 0$, which can be easily checked.

(ii) For $ 0.330384 \le \lambda < 9/25$, choose $I=[1-\lambda, 1-\lambda+\lambda^3]$, and $J=[1-\lambda+\lambda^2 - \lambda^3, 1-\lambda+\lambda^2]$.
That is, $n=3$, $a=1-\lambda$, and $b= 1-\lambda+\lambda^2 - \lambda^3$ in Lemma \ref{lem:find-nontrivial}.
We clearly have $b+\lambda^3 < 1 < a/(1-2\lambda)$.
By Lemma \ref{lem:find-nontrivial}, we need
$$ \frac{1-2\lambda}{\lambda} (1-\lambda+ \lambda^3) \le 1-\lambda+\lambda^2 - \lambda^3,$$
i.e., $\lambda^4 - 3\lambda^2 + 4 \lambda -1 \ge 0$, which can be checked directly.
Note that $(1-\lambda+\lambda^3)^2 + (1-\lambda+\lambda^2)^2 > (1-\lambda)^2 + (1-\lambda+\lambda^2)^2 >1$ for $0< \lambda \le 9/25$.
In order to show $1 \in g(I\times J)$, it remains to check
\begin{equation}\label{eq:desired-condition}
  (1-\lambda)^2+(1-\lambda+\lambda^2 - \lambda^3)^2 \le 1.
\end{equation}
Observe that the left-hand side in (\ref{eq:desired-condition}) is decreasing for $0< \lambda <1/2$.
Thus, one can easily verify (\ref{eq:desired-condition}) by plugging $\lambda =0.330384$ into the inequality.
\end{proof}

Finally, we will prove Theorem \ref{circle-lambda} (iii) by using the binary branching argument.
For $I,J \in \mathcal{F}_n$, we define \emph{the set of double covering points} $\mathcal{G}(I,J)$ to be the set of $z \in g(I\times J)$ satisfying that there exists $k > n$ and two distinct pairs $(I_1,J_1),(I_2,J_2) \in \mathcal{F}_k^I \times \mathcal{F}_k^J$ such that $z$ is an interior point of both $g(I_1\times J_1)$ and $g(I_2\times J_2)$.
This means that the point $z$ can be covered at least twice by the images of basic subintervals.
Under some conditions, we will show that the set of double covering points $\mathcal{G}(I,J)$ is the whole interval $g(I\times J)$ except for endpoints.

\begin{lemma}\label{lem:G-1}
Let $I,J \in \mathcal{F}_n$ with $I=[a,a+\lambda^n]$, $J=[b, b+\lambda^n]$, and $a\le b$.
Suppose that
\begin{equation}\label{eq:desired-condition-2}
  \frac{1-2\lambda}{\lambda} (a+\lambda^n) \le  b < b+\lambda^n \le \frac{\lambda a}{1-2\lambda}.
\end{equation}
Then we have
\[ \big( \alpha_n(a,b), \beta_n(a,b) \big) \subset \mathcal{G}(I,J), \]
where $$\alpha_n(a,b) := a^2+b^2 + 2a(1-\lambda)\lambda^n + (1-\lambda)^2 \lambda^{2n} ,$$
and $$\beta_n(a,b) := a^2+b^2 + 2(a\lambda+b)\lambda^n + (1+\lambda^2) \lambda^{2n}.$$
\end{lemma}
\begin{proof}
Let $I_1,I_2,J_1,J_2$ be defined as in Lemma \ref{lem:check-condition}.
By Lemma \ref{lem:check-condition} and (\ref{eq:desired-condition-2}), the four intervals $g(I_1\times J_1), g(I_2\times J_1),g(I_1\times J_2),g(I_2\times J_2)$ are overlapping.
Moreover, we want to show that
\[g(I_1\times J_1)\cap g(I_1\times J_2)\ne\emptyset\quad\textrm{and}\quad g(I_2\times J_1)\cap g(I_2\times J_2)\ne\emptyset,\] which is equivalent to
\begin{equation}\label{eq:equivalent-condition-2}
  \left\{
    \begin{array}{l}
      -(\lambda-1)^2\lambda^n \le 2a\lambda - 2(b+\lambda^n)(1-2\lambda),  \\
      (1-4\lambda+\lambda^2)\lambda^n \le 2a \lambda - 2b(1-2\lambda).
    \end{array}
  \right.
\end{equation}
Note that the left-hand sides of both two inequalities in (\ref{eq:equivalent-condition-2}) are negative for $2-\sqrt{3} < \lambda < 1/2$. Thus, we only need
$$a\lambda - (b+\lambda^n)(1-2\lambda)\ge 0,$$
which follows directly from (\ref{eq:desired-condition-2}).

\begin{figure}[h!]
\begin{center}
\begin{tikzpicture}[
    scale=10,
    axis/.style={very thick, ->},
    important line/.style={very thick},
    every node/.style={color=black}
    ]

    \draw[important line] (0, 0)--(0.3, 0);
    \node[above,scale=1pt]at(0.1,-0.06){$g(I_1\times J_1)$};

    \draw[important line] (0.15, 0.01)--(0.5,0.01);
    \node[above,scale=1pt]at(0.33,0.02){$g(I_2\times J_1)$};

    \draw[important line] (0.25, -0.01)--(0.6,-0.01);
    \node[above,scale=1pt]at(0.41,-0.07){$g(I_1\times J_2)$};

    \draw[important line] (0.45, 0)--(0.85,0);
    \node[above,scale=1pt]at(0.65,0.01){$g(I_2\times J_2)$};
\end{tikzpicture}
\end{center}
\caption{The relative position of intervals $g(I_1\times J_1)$, $g(I_2\times J_1)$, $g(I_1\times J_2)$ and $g(I_2\times J_2)$}\label{fig:2}
\end{figure}
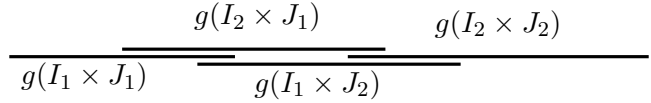
Now, the relative position of intervals $g(I_1\times J_1), g(I_2\times J_1),g(I_1\times J_2)$ and $g(I_2\times J_2)$ is illustrated as in Figure \ref{fig:2}.
By the definition of the set of double covering points, we conclude that
\[
(\alpha_n(a,b), \beta_n(a,b))=\mathrm{int}\big(g(I_2\times J_1)\cup g(I_1\times J_2)\big)\subset \mathcal G(I, J),
\]
completing the proof.
\end{proof}

\begin{lemma}\label{lem:G-2}
Let $I,J \in \mathcal{F}_n$ with $I=[a,a+\lambda^n]$, $J=[b, b+\lambda^n]$, and $a\le b$.
Suppose that
\begin{equation}\label{eq:desired-condition-3}
\frac{1-\lambda-\lambda^2}{\lambda} (a+\lambda^n) \le  b < b+\lambda^n \le \frac{\lambda a}{1-2\lambda}.
\end{equation}
Then we have
\[ \big( a^2+b^2, \beta_n(a,b) \big) \subset \mathcal{G}(I,J), \]
where $\beta_n(a,b)$ is defined as in Lemma \ref{lem:G-1}.
\end{lemma}
\begin{proof}
Note that for any $I' \in \mathcal{F}_{k}^{I}, J' \in \mathcal{F}_{k}^{J}$ with $k \ge n$, we have $\mathcal{G}(I',J') \subset \mathcal{G}(I,J)$.
For $k \ge n$, let $I'_k =[a,a+\lambda^k]$ and $J'_k=[b,b+\lambda^k]$.
Then $I'_k \in \mathcal{F}_k^{I}$ and $J'_k \in \mathcal{F}_{k}^{J}$.
Thus we have $$\bigcup_{k=n}^\f \mathcal{G}(I'_k,J'_k) \subset \mathcal{G}(I,J).$$
Note that (\ref{eq:desired-condition-3}) implies (\ref{eq:desired-condition-2}). For each $k \ge n$, by Lemma \ref{lem:G-1} we obtain $\big( \alpha_{k}(a,b), \beta_k(a,b)  \big) \subset \mathcal{G}(I'_k,J'_k)$.
It follows that
\begin{equation}\label{eq:G-left}
  \bigcup_{k=n}^\f \big( \alpha_{k}(a,b), \beta_k(a,b)  \big) \subset \mathcal{G}(I,J).
\end{equation}
Note that $\alpha_k(a,b)\searrow a^2+b^2$ and $\beta_k(a,b)\searrow a^2+b^2$ as $k\to\f$. Moreover, $\alpha_k(a,b)<\beta_{k+1}(a,b)$ for all $k\ge n$, since
by calculation we have
\begin{align*}
  \lambda^{-k} \big( \beta_{k+1}(a,b) - \alpha_k(a,b) \big) &= 2b\lambda - 2(a+\lambda^k)(1-\lambda-\lambda^2) + (\lambda^2 -1)^2\lambda^k \\
& > 2b\lambda - 2(a+\lambda^k)(1-\lambda-\lambda^2) \\
& \ge 0,
\end{align*}
where the last inequality follows directly from (\ref{eq:desired-condition-3}). 
Therefore, by (\ref{eq:G-left}) we conclude that \[ \big( a^2+b^2, \beta_n(a,b) \big) \subset \mathcal{G}(I,J), \] as desired.
\end{proof}

\begin{lemma}\label{lem:G-3}
Let $I,J \in \mathcal{F}_n$ with $I=[a,a+\lambda^n]$, $J=[b, b+\lambda^n]$, and $a\le b$.
Suppose that (\ref{eq:desired-condition-3}) in Lemma \ref{lem:G-2} holds.
Then we have
\[ \mathcal{G}(I,J) = \big( a^2+b^2, a^2+b^2+2(a+b)\lambda^n+2\lambda^{2n} \big)={\rm int}\big( g(I\times J) \big). \]
\end{lemma}
\begin{proof}
For $k \ge n$, let $\widetilde{I}'_k =[a_k,a_k+\lambda^k]$ and $\widetilde{J}'_k=[b_k,b_k+\lambda^k]$, where $a_k = a+\lambda^n -\lambda^k$ and $b_k = b+\lambda^n -\lambda^k$.
Note that $\widetilde{I}'_k  \in \mathcal{F}_k^{I}$ and $\widetilde{J}'_k \in \mathcal{F}_{k}^{J}$.
Thus we have $$\bigcup_{k=n}^\f \mathcal{G}(\widetilde{I}'_k,\widetilde{J}'_k) \subset \mathcal{G}(I,J).$$
By (\ref{eq:desired-condition-3}), we have
$$\frac{1-\lambda-\lambda^2}{\lambda} (a_k+\lambda^k) = \frac{1-\lambda-\lambda^2}{\lambda} (a+\lambda^n) \le  b \le b_k, $$
and $$b_k +\lambda^k = b+\lambda^n\le \frac{\lambda a}{1-2\lambda} \le \frac{\lambda a_k}{1-2\lambda}. $$
Applying Lemma \ref{lem:G-2} for $\widetilde{I}'_k ,\widetilde{J}'_k$, we obtain that $\big( a_k^2+b_k^2, \beta_k(a_k ,b_k)  \big) \subset \mathcal{G}(\widetilde{I}'_k,\widetilde{J}'_k)$.
It follows that
\begin{equation}\label{eq:G-right}
  \bigcup_{k=n}^\f \big( a_k^2+b_k^2, \beta_k(a_k ,b_k)  \big) \subset \mathcal{G}(I,J).
\end{equation}
Note that
\begin{align*}
  &\; \lambda^{-k}\big( \beta_k(a_k ,b_k) - a_{k+1}^2 - b_{k+1}^2 \big) \\
= &\; 2b\lambda - 2a(1-2\lambda) + 2(3\lambda-1)\lambda^n +(1-2\lambda-\lambda^2)\lambda^k \\
= &\; 2b\lambda - 2(a+\lambda^n)(1-2\lambda) + 2(\lambda^{n+1}-\lambda^{k+1}) +(1-\lambda^2)\lambda^k
\\
> & \; 2b\lambda - 2(a+\lambda^n)(1-\lambda-\lambda^2) \\
\ge & \; 0,
\end{align*}
where the last inequality follows from (\ref{eq:desired-condition-3}).
This implies that $\beta_k(a_k ,b_k) > a_{k+1}^2 + b_{k+1}^2$ for each $k \ge n$.
Note that $a_n = a$, $b_n= b$, and the sequence $\{\beta_k(a_k ,b_k)\}_{k=n}^\f$ is increasing to $(a+\lambda^n)^2+(b+\lambda^n)^2$ as $k \to \f$.
By (\ref{eq:G-right}), we conclude that
$$ \big( a^2+b^2, a^2+b^2+2(a+b)\lambda^n+2\lambda^{2n} \big) \subset \mathcal{G}(I,J). $$
By definition, the inverse inclusion is obvious. Thus we obtain the desired result.
\end{proof}

\begin{proposition}\label{prop:continuum}
Let $I,J \in \mathcal{F}_n$ with $I=[a,a+\lambda^n]$, $J=[b, b+\lambda^n]$, and $a<b$.
Suppose that (\ref{eq:desired-condition-3}) in Lemma \ref{lem:G-2} holds.
Then for any $a^2+b^2< r < a^2+b^2+2(a+b)\lambda^n+2\lambda^{2n}$, the intersection $\big\{(x,y):x^2+y^2= r \big\}\cap (K_\lambda\times K_\lambda)$ is of cardinality continuum.
\end{proposition}
\begin{proof}
For $k \ge n$ and for $I'=[a',a'+\lambda^k] \in \mathcal{F}_k^{I}$ and $J'=[b',b'+\lambda^k] \in \mathcal{F}_k^{J}$, noting that $a\le a'\le a+\lambda^n -\lambda^k$ and $b\le b' \le b+\lambda^n -\lambda^k$, by (\ref{eq:desired-condition-3}) we have
$$\frac{1-\lambda-\lambda^2}{\lambda} (a'+\lambda^k) \le  b' < b'+\lambda^k \le \frac{\lambda a'}{1-2\lambda}. $$
By Lemma \ref{lem:G-3}, we conclude that
\begin{equation}\label{eq:induction-G}
\mathcal{G}(I',J')=\big( a'^2+b'^2, a'^2 + b'^2 + 2(a'+b')\lambda^k + 2 \lambda^{2k}  \big) = \mathrm{int}\big( g(I'{\times} J') \big).
\end{equation}

Fix $a^2+b^2< r < a^2+b^2+2(a+b)\lambda^n+2\lambda^{2n}$.
Next, we will inductively define two nested sequences of basic intervals $\set{I_{i_1\ldots i_n}}_n$ and $\set{J_{i_1\ldots i_n}}_n$ with each $i_1\ldots i_n \in \{1,2\}^n$ for $n \in \N$ such that
\begin{itemize}
  \item $I_{i_1\ldots i_n}, J_{i_1\ldots i_n}$ are basic intervals of $K_\lambda$ at the same level;
  \item $I_{i_1\ldots i_n 1}, I_{i_1\ldots i_n 2} \subset I_{i_1\ldots i_n}$;
  \item $J_{i_1\ldots i_n 1}, J_{i_1\ldots i_n 2} \subset J_{i_1\ldots i_n}$;
  \item $r \in \mathcal{G}(I_{i_1\ldots i_n} , J_{i_1\ldots i_n})$.
\end{itemize}
By Lemma \ref{lem:G-3}, we have $r \in \mathcal{G}(I,J)$. Then we can find two distinct pairs $(I_1,J_1),(I_2,J_2) \in \mathcal{F}_k^I \times \mathcal{F}_{k}^J$ for some $k \in \N$ such that $r$ is an interior point of both $g(I_1\times J_1)$ and $g(I_2 \times J_2)$.
By (\ref{eq:induction-G}), we have $r \in \mathcal{G}(I_1,J_1)$ and $r \in \mathcal{G}(I_2,J_2)$.
Assume that $I_{i_1\ldots i_n}, J_{i_1\ldots i_n}$ have been defined.
Since $r \in \mathcal{G}(I_{i_1\ldots i_n}, J_{i_1\ldots i_n})$, we can find two distinct pairs $$(I_{i_1\ldots i_n 1},J_{i_1\ldots i_n 1}),(I_{i_1\ldots i_n 2},J_{i_1\ldots i_n 2}) \in \mathcal{F}_{k'}^{I_{i_1\ldots i_n}} \times \mathcal{F}_{k'}^{J_{i_1\ldots i_n}}$$ for some $k' \in \N$ such that $r$ is an interior point of both $g(I_{i_1\ldots i_n 1} \times J_{i_1\ldots i_n 1})$ and $g(I_{i_1\ldots i_n 2} \times J_{i_1\ldots i_n 2})$.
By (\ref{eq:induction-G}), we have $r \in \mathcal{G}(I_{i_1\ldots i_n 1},J_{i_1\ldots i_n 1})$ and $r \in \mathcal{G}(I_{i_1\ldots i_n 2},J_{i_1\ldots i_n 2})$.

For any sequence $\mathbf{i} = i_1 i_2 \ldots \in \{1,2\}^\N$, we define $x_{\mathbf{i}}$ and $y_{\mathbf{i}}$ to be the unique point in $$\bigcap_{n=1}^\f I_{i_1 \ldots i_n} \quad \text{and} \quad \bigcap_{n=1}^\f J_{i_1 \ldots i_n},$$ respectively.
Then we have $x_{\mathbf{i}},y_{\mathbf{i}} \in K_\lambda$.
Note that $$r \in \mathcal{G}(I_{i_1\ldots i_n} , J_{i_1\ldots i_n}) \subset g(I_{i_1\ldots i_n} \times J_{i_1\ldots i_n}).$$
By Lemma \ref{lem:continuity}, we have
$$\{r\} \subset \bigcap_{n=1}^\f g(I_{i_1\ldots i_n} \times J_{i_1\ldots i_n}) = g\Big( \bigcap_{n=1}^\f I_{i_1 \ldots i_n} \times \bigcap_{n=1}^\f J_{i_1 \ldots i_n} \Big) =\{x_{\mathbf{i}}^2 + y_{\mathbf{i}}^2\}.$$
That is, $x_{\mathbf{i}}^2 + y_{\mathbf{i}}^2 = r$.
It follows that $$\big\{(x_{\mathbf{i}},y_{\mathbf{i}}): \mathbf{i} \in \{1,2\}^\N \big\} \subset \big\{(x,y):x^2+y^2= r \big\}\cap (K_\lambda\times K_\lambda).$$
Note that any two points $(x_{\mathbf{i}},y_{\mathbf{i}}), (x_{\mathbf{i}'},y_{\mathbf{i}'})$ for $\mathbf{i}\ne \mathbf{i}' \in \{1,2\}^\N$ are different.
Thus we conclude that the intersection $\big\{(x,y):x^2+y^2= r \big\}\cap (K_\lambda\times K_\lambda)$ is of cardinality continuum.
\end{proof}

\begin{proposition}\label{prop:circle-3}
For $0.407493 \le \lambda < 1/2$, the intersection $\mathbb{S}\cap (K_\lambda\times K_\lambda)$ is of cardinality continuum.
\end{proposition}
\begin{proof}
(i) For $0.415 \le \lambda < 1/2$, let $I=[1-\lambda+\lambda^2-\lambda^3, 1-\lambda+\lambda^2-\lambda^3+\lambda^4]$ and $J=[1-\lambda+\lambda^2-\lambda^4, 1-\lambda + \lambda^2]$. That is, $n=4$, $a = 1-\lambda+\lambda^2-\lambda^3$, and $b=1-\lambda+\lambda^2-\lambda^4$.
In order to apply Proposition \ref{prop:continuum}, we need to verify the following inequalities.
\begin{equation}\label{eq:computer-1}
\left\{
    \begin{array}{l}
      \frac{1-\lambda-\lambda^2}{\lambda}(1-\lambda+\lambda^2-\lambda^3+\lambda^4) \le 1-\lambda+\lambda^2-\lambda^4,\\
      1-\lambda + \lambda^2 \le \frac{\lambda}{1-2\lambda}(1-\lambda+\lambda^2-\lambda^3),\\
      (1-\lambda+\lambda^2-\lambda^3)^2 + (1-\lambda+\lambda^2-\lambda^4)^2< 1,  \\
      (1-\lambda+\lambda^2-\lambda^3+\lambda^4)^2 + (1-\lambda+\lambda^2)^2> 1.
    \end{array}
  \right.
\end{equation}

(ii) For $0.407494 \le \lambda < 0.415$, let $I=[1-\lambda, 1-\lambda+\lambda^3]$ and $J=[1-\lambda+\lambda^2-\lambda^3, 1-\lambda + \lambda^2]$. That is, $n=3$, $a = 1-\lambda $, and $b=1-\lambda+\lambda^2-\lambda^3$.
In order to apply Proposition \ref{prop:continuum}, we need to verify the following inequalities.
\begin{equation}\label{eq:computer-2}
\left\{
    \begin{array}{l}
      \frac{1-\lambda-\lambda^2}{\lambda}(1-\lambda+\lambda^3) \le 1-\lambda+\lambda^2-\lambda^3,\\
      1-\lambda + \lambda^2 \le \frac{\lambda}{1-2\lambda}(1-\lambda),\\
      (1-\lambda)^2 + (1-\lambda+\lambda^2-\lambda^3)^2< 1,  \\
      (1-\lambda+\lambda^3)^2 + (1-\lambda+\lambda^2)^2> 1.
    \end{array}
  \right.
\end{equation}

The inequalities in (\ref{eq:computer-1}) and (\ref{eq:computer-2}) can be easily checked with the assistance of computers.
The desired result follows from Proposition \ref{prop:continuum}.
\end{proof}


\begin{proof}[Proof of Theorem \ref{circle-lambda}]
It follows directly from Propositions \ref{prop:circle-1}, \ref{prop:circle-2}, and \ref{prop:circle-3}.
\end{proof}

\section{Intersection with a curve  or a surface}\label{sec:curve}

Recall that $K_\lambda$ is a self-similar set generated by the IFS
$\big\{ f_0(x) = \lambda x ,\; f_1(x) = \lambda x + 1- \lambda \big\}$ for $\lambda \in (0,1/2)$.
Each point $x \in K_\lambda$ corresponds to a unique sequence $(x_n) \in \{0,1\}^\N$, called the \emph{coding of $x$}, such that
$$x= (1-\lambda)\sum_{n=1}^{\f} x_n \lambda^{n-1}.$$
We will use the following Bernoulli inequality: for $n \in \N_{\ge 2}$,
\begin{equation}\label{Bernoulli-inequality}
  (1+x)^n > 1+nx, \quad \forall x > -1 \text{ with } x \ne 0.
\end{equation}
In this section, we always assume that $0< \lambda \le 1/3$, and $q \in \N$.

\begin{lemma}\label{key}
Let $x,y\in K_\lambda$, and let $(x_n),(y_n) \in \{0,1\}^{\N}$ be the codings of $x$ and $y$, respectively. Suppose $y=x^{q}$ for some $2 \le q \le 1/\lambda - 1$.
If $x_1x_2\ldots x_kx_{k+1}=0^k1$ for some $k \in \N$, then $y_1y_2\ldots y_{q k}y_{q k+1}=0^{q k}1$.
\end{lemma}
\begin{proof}
Since $x_1x_2\ldots x_k=0^k$, we have $$x= (1-\lambda)\sum_{n=k+1}^{\f}x_n\lambda^{n-1} \le \lambda^k.$$
Then $y = x^q \le \lambda^{q k}$.
This implies that $y_1y_2\ldots y_{q k}=0^{q k}$.

Note that  $x_{k+1}=1$. Then $$x= (1-\lambda)\sum_{n=k+1}^{\f}x_n\lambda^{n-1} \ge (1-\lambda)\lambda^k.$$
By Bernoulli inequality (\ref{Bernoulli-inequality}) and using $2\le q\le 1/\lambda -1$, we have $$y = x^q \ge (1-\lambda)^q \lambda^{qk} > (1- q\lambda)\lambda^{qk} \ge  \lambda^{qk+1} .$$
This together with $y_1y_2\ldots y_{q k}=0^{q k}$ implies $y_{qk +1} = 1$.
\end{proof}

\begin{lemma}\label{key1}
Let $x,y\in K_\lambda$, and suppose that $y=x^{q}$ for some $2 \le q \le 1/\lambda - 1$.
If $x,y\in [1-\lambda,1]\cap K_\lambda$, then $x=y=1$.
\end{lemma}
\begin{proof}
It suffices to show that $x,y\in [1-\lambda^\ell,1]\cap K$ for any $\ell \in \N$.
Clearly, the conclusion holds for $\ell =1$.
Next, suppose that $x,y\in [1-\lambda^k,1] \cap K_\lambda$ for some $k \in \N$, we show that $x,y\in [1-\lambda^{k+1},1]\cap K_\lambda$.
If this were done, we would complete the proof by induction.

Observe that $$x,y \in \big[ 1-\lambda^k, 1 \big] \cap K_\lambda \subset \big[ 1-\lambda^k, 1-(1-\lambda)\lambda^k \big] \cup \big[ 1-\lambda^{k+1}, 1 \big]. $$
If $x \le 1-(1-\lambda)\lambda^k$, then
\begin{equation*}
\begin{split}
  y & \le \Big( 1-(1-\lambda)\lambda^k \Big)^q \le \Big( 1-(1-\lambda)\lambda^k \Big)^2 \\
  & = 1- 2(1-\lambda) \lambda^k + (1-\lambda)^2\lambda^{2k} \\
  & \le 1 - 2(1-\lambda) \lambda^k + (1-\lambda)^2\lambda^{k+1} \\
  & \le 1- \lambda^k -(1-3\lambda + 2 \lambda^2 - \lambda^3)\lambda^k \\
  & < 1- \lambda^k.
  \end{split}
\end{equation*}
where the last inequality follows from $0<\lambda \le 1/3$.
This leads to a contradiction.
Thus, we obtain $x \in [1-\lambda^{k+1},1]\cap K_\lambda$.
Then, by Bernoulli inequality (\ref{Bernoulli-inequality}) and using $2\le q\le 1/\lambda -1$ we have
$$y = x^q\ge (1-\lambda^{k+1})^q > 1- q\lambda^{k+1} \ge 1- (1-\lambda)\lambda^k.$$
This implies that $y\in [1-\lambda^{k+1},1]\cap K_\lambda$.
Therefore, we conclude that $x,y\in [1-\lambda^{k+1},1]\cap K$.
\end{proof}

\begin{proposition}\label{prop:curve-1}
Suppose that $2 \le q \le 1/\lambda -1$. Then we have
$$\{(x,y):y=x^{q}\}\cap (K_\lambda \times K_\lambda) =\big\{ (\lambda^k,\lambda^{qk} ): k\in \mathbb{N}\big\}\cup\{(0,0), (1,1)\}. $$
\end{proposition}
\begin{proof}
Take $x,y\in K_\lambda \setminus \{0\}$ with $y=x^{q}$.
Let $(x_n),(y_n) \in \{0,1\}^{\N}$ be the codings of $x$ and $y$, respectively.

If $x_1 = 1$, then $x \ge 1-\lambda$. By Bernoulli inequality (\ref{Bernoulli-inequality}) and using $2\le q\le 1/\lambda -1$, we have $$y = x^q \ge (1-\lambda)^q > 1- q \lambda \ge \lambda.$$
This implies that $y_1 = 1$.
Thus we have $x,y \in [1-\lambda,1]\cap K_\lambda$.
Hence, by Lemma \ref{key1} we conclude that $x=y=1$.

If $x_1 =0$, then since $x\ne 0$ there exists $k \in \N$ such that $$x_1 x_2 \cdots x_k = 0^k \text{ and } x_{k+1} = 1.$$
Then by Lemma \ref{key} we have
$$y_1 y_2 \cdots y_{q k} = 0^{q k} \text{ and } y_{q k+1} = 1.$$
Let $\tilde{x} = \lambda^{-k} x$ and $\tilde{y} = \lambda^{-q k} y$.
Then $\tilde{x},\tilde{y} \in [1-\lambda,1] \cap K_\lambda$, and $\tilde{y} = \tilde{x}^q$.
Hence, by Lemma \ref{key1} we conclude that $\tilde{x}=\tilde{y}=1$.
That is, $x= \lambda^k$ and $y = \lambda^{q k}$.
This completes the proof.
\end{proof}

\begin{proposition}\label{prop:curve-2}
For $0< \lambda \le 0.187915$, we have $$\{(x,y,z): x^2+y^2=z^2\}\cap (K_\lambda\times K_\lambda \times K_\lambda) =\{(x,0,x), (0,x,x): x \in K_\lambda\}. $$
\end{proposition}
\begin{proof}
The ideas are similar to the proof of Proposition \ref{prop:curve-1}.
Let $x,y,z \in K_\lambda$ with $x^2 + y^2  =  z^2$. By symmetry, we can assume $x\le y$.
We will show that $x=0$.
If $z=0$, then $x=y=z=0$. In the following, we assume that $z\ne 0$.

We first assume that $z \in [1-\lambda,1] \cap K_\lambda$.
If $x \ge 1-\lambda$, then using $0< \lambda < 1-\sqrt{2}/2$ we have $$ x^2 + y^2 \ge 2x^2 \ge 2(1- \lambda)^2 > 1 \ge z^2.$$
Thus, we conclude that $x \in [0,\lambda]$.
Next, suppose that $x\in [0,\lambda^k]$ for some $k\in \N$. We shall show $x\in [0,\lambda^{k+1}]$. If this were done, then we could conclude $x\in [0,\lambda^\ell]$ for any $\ell \in \N$, which implies that $x=0$.

Observe that
$$x \in [0,\lambda^k] \cap K_\lambda \subset [0,\lambda^{k+1}] \cup [(1-\lambda)\lambda^k,\lambda^k]. $$
Suppose on the contrary that $x \in [(1-\lambda)\lambda^k,\lambda^k]$.
If $y \ge z-\lambda^{2k+1}$, then we have
\begin{align*}
x^2+y^2 & \geq (1-\lambda)^2 \lambda^{2k} + (z-\lambda^{2k+1})^2 \\
& > z^2 + (\lambda^2-2\lambda+1 - 2z\lambda) \lambda^{2k} \\
& \ge z^2 + (\lambda^2-4\lambda+1) \lambda^{2k}\\
& > z^2,
\end{align*}
where the last inequality follows from $\lambda < 2 - \sqrt{3}$.
If $y < z-\lambda^{2k+1}$, then let $(y_n),(z_n) \in \{0,1\}^{\N}$ be the codings of $y$ and $z$, respectively. That is, $$y =(1-\lambda)\sum_{n=1}^{\f} x_n \lambda^{n-1} \quad \text{ and} \quad z = (1-\lambda)\sum_{n=1}^{\f} z_n \lambda^{n-1}.$$
Since $y< z$, there exists $\ell \in \N$ such that
$$y_n = z_n \text{ for } 1\le n \le \ell-1, \quad \text{and} \quad y_\ell < z_\ell.$$
Then we have $y_\ell =0$ and $z_\ell = 1$.
Observe that $$\lambda^{2k+1} < z - y \le (1-\lambda)\sum_{n=\ell}^{\f} z_n \lambda^{n-1} \le \lambda^{\ell-1}.$$
This implies that $\ell \le 2k+1$.
It follows that
\begin{align*}
z-y & \ge (1-\lambda)\lambda^{\ell-1}  - (1-\lambda)\sum_{n=\ell+1}^{\f} y_n \lambda^{n-1} \\
& \ge (1-\lambda)\lambda^{\ell-1} -\lambda^{\ell} \\
& \ge (1-2\lambda)\lambda^{2k},
\end{align*}
i.e., $y \le z - (1-2\lambda)\lambda^{2k}$.
Thus using $z \ge 1-\lambda$ we have
\begin{align*}
  x^2 + y^2 & \le \lambda^{2k} + \Big( z - (1-2\lambda)\lambda^{2k} \Big)^2 \\
  & = z^2 - \Big( 2(1-2\lambda)z-1 \Big) \lambda^{2k} + (1-2\lambda)^2 \lambda^{4k} \\
  & \le z^2 - \Big( 2(1-2\lambda)(1-\lambda) - 1 - (1-2\lambda)^2 \lambda^2 \Big) \lambda^{2k} \\
  & = z^2 - ( 1-6\lambda+3\lambda^2 + 4\lambda^3 - 4\lambda^4 ) \lambda^{2k} \\
  & < z^2,
\end{align*}
where the last inequality follows from $0< \lambda \le 0.187915$.
This leads to a contradiction with $x^2 + y^2 = r^2$.
Thus, we conclude that $x \in [0,\lambda^{k+1}]$.

Next, if $z \in [0,\lambda]\cap K_\lambda$, then since $z\ne 0$ there exists $k \in \N$ such that $z \in [(1-\lambda)\lambda^k,\lambda^k]$.
Let $\tilde{x} = \lambda^{-k} x$, $\tilde{y} = \lambda^{-k} y$, and $\tilde{z} = \lambda^{-k} z$.
Then we have $\tilde{x},\tilde{y},\tilde{z} \in K_\lambda$ and $\tilde{x}^2 + \tilde{y}^2 - \tilde{z}^2 =0$.
Note that $\tilde{z} \in [1-\lambda,1]\cap K_\lambda$.
By the above arguments, we have $\tilde{x} = 0$.
This implies that $x =0$.

Therefore, we conclude that either $x=0,y=z \in K_\lambda$ or $y=0, x= z \in K_\lambda$.
\end{proof}

\begin{proof}[Proof of Theorem \ref{curve}]
It follows directly from Propositions \ref{prop:curve-1} and \ref{prop:curve-2}.
\end{proof}

\section{Intersection with a sequence}\label{sec:sequence}

We first introduce the quadratic residues and the Legendre's symbol. Let $p\in \N$ be an odd prime and let $a \in \mathbb{Z}$ with $p \nmid a$.
We say that $a$ is a \emph{quadratic residue} of $p$ if the congruence $x^2 \equiv a \pmod{p}$ has a solution in $\{1,2,\ldots,p-1\}$; otherwise, $a$ is called a \emph{quadratic non-residue} of $p$ \cite{Hardy2008}.
The Legendre's symbol is defined by
\[\left(\dfrac{a}{p}\right)_L=
\begin{cases}
1 & \text{if}\; a \;\text{is a quadratic residue of}\; p
\\-1 & \text{if}\; a \;\text{is a quadratic non-residue of}\; p.
\end{cases}\]

\begin{proof}[Proof of Theorem \ref{sequence}]
Since $(m-1)^2 \equiv 1 \pmod{m}$, we have
$$\left(\dfrac{-(m-1)}{m}\right)_L=\left(\dfrac{1}{m}\right)_L =1.$$
This means that $m-1 \not\in D$.
It follows that $1 \not\in K_{m,D}$.

We first show that for $n \in \N_{\ge 2}$, if $m \nmid n$ then  $1/n^{2} \not\in K_{m,D}$.
Fix $n \in \N_{\ge 2}$ with $m \nmid n$.
Since $m$ is prime, we have $\gcd(n^{2},m)=1$.
By \cite[Proposition 2.1.2]{Karma2002}, the expansion of $1/n^{2}$ in base $m$ is purely periodic, denoted by $(x_1 x_2 \ldots x_q)^\f$, where $x_1, x_2, \ldots, x_q \in \{0,1,\ldots, m-1\}$.
Then we have
$$
\frac{1}{n^{2}} = \bigg( \sum_{i=1}^{q} \frac{x_i}{m^i} \bigg)\bigg/ \bigg( 1- \frac{1}{m^q} \bigg).
$$
That is,
$$
(x_q + x_{q-1} m + \cdots + x_1 m^{q-1})n^{2} = m^q-1.
$$
Taking modulo $m$ on the both sides, we obtain
\begin{equation}\label{congruence}
 x_q n^{2} \equiv -1 \pmod{m}.
\end{equation}
Thus, $x_q\neq 0$ and it follows that
$$(x_q n)^{2} \equiv -x_q \pmod{m},$$
which yields that $$\left(\dfrac{-x_q}{m}\right)_L =1.$$
Thus, we conclude that $x_q \not\in D$.
Observe that $(x_1 x_2 \cdots x_q)^\f$ is the unique expansion of $1/n^{2}$ in base $m$. Therefore, we obtain $1/n^{2} \not\in K_{m,D}$.

Next, we show that
\begin{equation}\label{eq:june-5}
\left\{\dfrac{1}{n^{2}}:n\in \mathbb{N}\right\}\cap K_{m,D} \subset \left\{\dfrac{1}{m^{2\ell}}: \ell \in \N\right\}.
\end{equation}
Take $n \in \N$ with $1/n^{2} \in K_{m,D}$.
By the above arguments, we have $n \ge 2$ and $m \mid n$.
Write $n= \widetilde{n} m^\ell$ with $\ell \in \N$ and $m \nmid \widetilde{n}$.
If $\widetilde{n} \ge 2$, then we have $1/(\widetilde{n})^{2} \not\in K_{m,D}$. However,
$$ \frac{1}{(\widetilde{n})^{2}} = m^{2\ell} \cdot \frac{1}{n^{2}}\in K_{m,D}, $$
leading to a contradiction.
Thus, we obtain $\widetilde{n}=1$, i.e., $n = m^\ell$ with $\ell \in \N$.

(1) If $1\not\in D$, then noting that $m-1 \not\in D$, we have $$\left\{\dfrac{1}{m^{2\ell}}: \ell \in \N\right\} \cap K_{m,D} = \emptyset.$$
Therefore,  we conclude that $$\left\{\dfrac{1}{n^{2}}:n\in \mathbb{N}\right\}\cap K_{m,D} =\emptyset.$$

(2) If $1\in D$, then we clearly have
$$\left\{\dfrac{1}{m^{2\ell}}: \ell \in \N\right\} \subset K_{m,D}.$$
Therefore, by (\ref{eq:june-5}) we conclude that $$\left\{\dfrac{1}{n^{2}}:n\in \mathbb{N}\right\}\cap K_{m,D} =\left\{\dfrac{1}{m^{2\ell}}: \ell \in \N\right\},$$
as desired.
\end{proof}

Next, we give some basic properties about quadratic residues, see \cite[Theorem 82, 83, 93]{Hardy2008}.

\begin{proposition}\label{prop:quadratic-residue}
Let $p$ be an odd prime. Then
\begin{enumerate}[{\rm(i)}]
\item $\left(\dfrac{-1}{p}\right)_L=(-1)^{\dfrac{p-1}{2}}$;
\item $\left(\dfrac{2}{p}\right)_L=(-1)^{\dfrac{p^2-1}{8}}$;
\item if $a, b\in\mathbb Z$ with $p\nmid a$ and $p\nmid b$, then  $\left(\dfrac{ab}{p}\right)_L = \left(\dfrac{a}{p}\right)_L \left(\dfrac{b}{p}\right)_L$.
\end{enumerate}
\end{proposition}

The most famous theorem in quadratic residues is Guass's law of reciprocity, see \cite[Theorem 98]{Hardy2008}.

\begin{theorem}[Quadratic Reciprocity Law]\label{qrl}
Let $p,q$ be two different odd primes. Then
$$\left(\dfrac{p}{q}\right)_L\left(\dfrac{q}{p}\right)_L=(-1)^{\dfrac{(p-1)(q-1)}{4}}.$$
\end{theorem}

%
%

%

 We utilize Theorem \ref{sequence} and properties of quadratic residue to obtain the following corollary. 

\begin{corollary}
For any $k \in \N_{\ge 2}$, there exist infinite many primes $m >k$ such that the Cantor set $K_{m,D}$  with digit set $D=\{0,1,\ldots,k\}$ satisfying
$$\left\{\dfrac{1}{n^{2}}:n\in \mathbb{N}\right\}\cap K_{m,D} =\left\{\dfrac{1}{m^{2\ell}}:\ell\in \mathbb{N}\right\}.$$
\end{corollary}
\begin{proof}
Let $p_1=2,p_2,\ldots,p_t$ be all primes less than or equal to $k$.
By Dirichlet's theorem, there are infinite many primes $m > k$ of the form $4np_1 p_2 \ldots p_t -1$ with $n \in \N$.
Fix a such prime $m = 4np_1 p_2 \ldots p_t -1$.

By Proposition \ref{prop:quadratic-residue} (i), we have
\begin{equation}\label{eq:Legendre-symbol-1}
  \left(\dfrac{-1}{m}\right)_L=-1.
\end{equation}
For $2 \le j \le t$, by the quadratic reciprocity law, we have
\begin{align*}
\left(\frac{p_j}{m}\right)_L&=\left(\frac{m}{p_j}\right)_L\cdot (-1)^{\frac{(p_j-1)(m-1)}{4}}\\
&=\left(\frac{-1}{p_j}\right)_L\cdot (-1)^{\frac{(p_j-1)(m-1)}{4}}\\
&=(-1)^{\frac{p_j-1}{2}}\cdot (-1)^{\frac{(p_j-1)(m-1)}{4}}\\
&= (-1)^{\frac{(p_j-1)(m+1)}{4}}=1.
\end{align*}
Note by Proposition \ref{prop:quadratic-residue} (ii) that
$$ \left(\dfrac{2}{m}\right)_L=1.$$
So, we have showed that \begin{equation}\label{eq:june5-1}
 \left(\frac{p_j}{m}\right)_L = 1,\quad \forall 1\le j \le t.
 \end{equation}
Each $2\le i \le k$ can be factored as $p_1^{n_1} p_2^{n_2} \ldots p_t^{n_t}$, where $n_1,n_2,\ldots,n_t \ge 0$.
Thus, by (\ref{eq:june5-1}) and Proposition \ref{prop:quadratic-residue} (iii), we obtain that
$$ \left(\frac{i}{m}\right)_L = 1,\quad \forall 2\le i \le k.$$
Again by Proposition \ref{prop:quadratic-residue} (iii) and by (\ref{eq:Legendre-symbol-1}), we have
$$ \left(\frac{-i}{m}\right)_L = -1,\quad \forall 1\le i \le k.$$
By Theorem \ref{sequence}, we conclude the desired result.
\end{proof}

%

\section{Some questions}\label{sec:question}

At the end of this paper, we list some questions we are interested in.
Let $C$ be the middle-third Cantor set.

\begin{question}
Is the intersection
$$\{(x,y):x^2+y^2=1\}\cap (C\times C)$$
infinite? Moreover, to determine the Hausdorff dimension of the intersection.
We expect the value to be $2 \log 2 / \log 3 -1$.
\end{question}

\begin{question}
What is the intersection
$$
\left\{\dfrac{1}{n^2}:n\in \mathbb{N}\right\}\cap C ?
$$
It's easy to check that
$$\left\{\dfrac{1}{9^k},\dfrac{1}{4\times 9^k},\dfrac{1}{121\times 9^k}:k\in \mathbb{N}\cup \{0\} \right\} \subset \left\{\dfrac{1}{n^2}:n\in \mathbb{N}\right\}\cap C.$$
Is there any other elements in the intersection?
\end{question}

\begin{question}
Do we have
$$
\left\{\dfrac{1}{n!}:n\in \mathbb{N}\right\}\cap C=\{1\}?
$$
\end{question}

\section*{Acknowledgements}
K.~Jiang was supported by Zhejiang Provincial Natural Science Foundation of China
under Grant No. LMS25A010009.
D.~Kong was supported by Chongqing NSF No.~CQYC20220511052.
W.~Li was supported by NSFC No.~12471085 and Science and Technology Commission of Shanghai Municipality (STCSM) No. 22DZ2229014.  
Z.~Wang was supported by the China Postdoctoral Science Foundation No.~2024M763857.

\end{document}